\documentclass[11pt,reqno]{article}
\usepackage{amssymb,amsmath,amsthm}
\usepackage[utf8]{inputenc}
\usepackage{tikz}
\usepackage{hyperref}
\usepackage{xcolor}
\usepackage[backend=biber,maxbibnames=99]{biblatex}
\usepackage{authblk} 

\usepackage{ulem}
\usepackage[shortlabels]{enumitem}
\setlist[enumerate]{topsep=0pt,itemsep=-1ex,partopsep=1ex,parsep=1ex}
\usepackage{graphicx}

\newtheorem{statement}{}[section]
\newtheorem{theorem}[statement]{Theorem}
\newtheorem{lemma}[statement]{Lemma}
\newtheorem{proposition}[statement]{Proposition}
\newtheorem{definition}[statement]{Definition}
\newtheorem{corollary}[statement]{Corollary}
\newtheorem{remark}[statement]{Remark}
\newtheorem{example}[statement]{Example}

\makeatletter
\newcommand{\subjclass}[2][1991]{%
  \let\@oldtitle\@title%
  \gdef\@title{\@oldtitle\footnotetext{#1 \textbf{Mathematics subject classification:} #2}}%
}
\newcommand{\keywords}[1]{%
  \let\@@oldtitle\@title%
  \gdef\@title{\@@oldtitle\footnotetext{\textbf{Key words and phrases:} #1}}%
}
\makeatother

\def\wgro{\ensuremath \gamma} 
\def\contop{\ensuremath \nu_{\text{op}}}

\def\ppx{\ensuremath \rho(x)}
\def\ppy{\ensuremath \rho(y)}
\def\dxy{\ensuremath d(x,y)}

\def\dxz{\ensuremath d(x,z)}
\def\dyz{\ensuremath d(y,z)}

\def\rx{\ensuremath R_y(x,0)}
\def\ry{\ensuremath R_x(y,0)}
\def\rxy{\ensuremath R_0(x,y)}
\def\rxz{\ensuremath R_y(x,z)}
\def\ryz{\ensuremath R_x(y,z)}
\def\rxyz{\ensuremath R_z(x,y)}

\def\groxy{\ensuremath \gro_0(x,y)}
\def\grox{\ensuremath \gro_y(x,0)}
\def\groy{\ensuremath \gro_x(y,0)}

\def\contopxy{\ensuremath \contop(x,y)}
\def\contopx{\ensuremath \contop(x,0)}
\def\contopy{\ensuremath \contop(y,0)}

\def\contopxz{\ensuremath \contop(x,z)}

\def\wgrox{\ensuremath \gamma(x,0)} 
\def\wgroy{\ensuremath \gamma(y,0)} 
\def\wgroxy{\ensuremath \gamma(x,y)}

\def\face{\ensuremath\mathfrak{F}}

\newcommand{\free}[1]{\mathcal{F}(#1)} 

\def\lip{ \text{Lip}}

\def\gro{ \ensuremath G }

\def\NN{\Bbb N}
\def\IN{\hbox{{\rm I}\kern-.13em{\rm N}}}

\def\RR{\mathbb{R}}
\def\IR{\hbox{{\rm I}\kern-.13em{\rm R}}}

\def\sp{\hbox{{\rm span}}}
\def\csp{\overline{\sp}}

\def\aa{\alpha}

\def\dd{\delta}

\def\pp{\rho}

\def\eor{\hfill{\circledR}}

\def\Ext{\hbox{\rm{Ext} }}

\def\la{\langle}
\def\ra{\rangle}

\title{The numerical index of $2$-dimensional Lipschitz-free spaces}

\subjclass[\textbf{2020}]{46B20, 
47A12, 
54E35 
}
\keywords{Lipschitz-free, numerical index, numerical radius, geometry of Banach spaces.}

\author{Ch. Cobollo\footnote{Christian Cobollo. 
Email: chcogo@upv.es. Corresponding author.
}
,
A. J. Guirao\footnote{Antonio José Guirao. 
Email: anguisa2@mat.upv.es.
}
,
and V. Montesinos\footnote{Vicente Montesinos. 
Email: vmontesinos@mat.upv.es.
}
}
\affil{\textit{Instituto Universitario de Matem\'atica Pura y Aplicada. Universitat Polit\`ecnica de Val\`encia (Spain).}}

\date{}

\begin{document}
\normalem
\maketitle
\begin{abstract}
We provide the explicit formula for the numerical index of any $2$-dimensional Lipschitz-free space, also giving the construction of operators attaining this value as its numerical radius. As a consequence, the numerical index of $2$-dimensional Lipschitz-free spaces can take any value of the interval $[\frac{1}{2},1]$, and this whole range of numerical indices can be attained by taking $2$-dimensional subspaces of any Lipschitz-free space of the form $\free{A}$, where $A\subset \RR^n$ with $n\geq 2$ is any set with non-empty interior.


\end{abstract}

\section{Introduction}
 The numerical index of a given Banach space is a constant relating the geometry of the norm and the numerical radius of bounded linear operators on the space. The concept appeared in the literature for the first time in 1970 in the seminal paper \cite{DMPW70}, after being introduced by G. Lumer two years before. The classical reference on this topic are the two monographs by F. Bonsall and J. Duncan, \cite{BonDun2} and \cite{BonDun1}. From there, a good overview of the current state of the topic can be provided through the survey \cite{survey-ind-num} together with the first chapter of the recent book \cite{spear-op}.
 
 This concept has been widely studied, and although the numerical index of certain classes of Banach spaces---like $L_1(\mu)$ or $C(K)$---are well known due to their concrete properties, finding the numerical index is still a challenging task, even for $2$-dimensional Banach spaces. For instance, there is a full line of recent papers devoted to the study of the numerical index of $2$-dimensional $l_p$ spaces, see  \cite{MarMer09}, \cite{MerQue21}, \cite{MerQue23} or \cite{MonZhe22}. Also, there are papers studying the numerical index for some $2$ or $3$-dimensional Banach spaces whose unit ball has a particular shape---see \cite{MarMer07} and \cite{3D-num-ind}.

Along this work, we intend to begin with the study of the numerical radius of Lipschitz-free Banach spaces, starting from the $2$-dimensional case. The origin of this class of Banach spaces---also known by other names through the literature, like Arens--Eells spaces or Transportation Cost spaces---can be tracked to classical authors like Kantorovich and Rubinstein \cite{Kant-Rub} while approaching optimal transport problems. They were rediscovered many times from different perspectives by authors like Arens, de Leeuw, Eells, or Johnson---see \cite{Ar-Ee}, \cite{Johnson}, and \cite{deLeeuw}. Meanwhile, Kadets \cite{Kadets} and Pestov \cite{Pestov} provided early versions of what nowadays is known as the universal property, as Lipschitz-free Banach space are the free object between the categories of metric spaces (with the morphisms of Lipschitz maps) and Banach spaces (with the morphisms of linear bounded operators). Its systematic study as Banach spaces started in 1999 with the first edition of the monograph \cite{Weaver1} by Weaver, whose recent second edition \cite{Weaver2} has taken its place as the main reference in the topic. Concerning the area of Geometry of Banach spaces, the seminal paper \cite{GoKa03} by Godefroy and Kalton in 2003 started the study of Lipschitz-free spaces from this point of view, providing some remarkable results, like the preservation of the bounded approximation property by Lipschitz-homeomorphisms, or the linearization of metric embeddings of a separable Banach space $X$ inside a Banach space $Y$.

As the norm of these spaces is fundamentally described as a predual norm, even nowadays, the geometry of these spaces is still under study, and it is unknown in so many aspects, being a very active field of research in the last years. For instance, it is still an open problem to determine if, in every Lipschitz-free space, every extreme point is a molecule (defined below).

\medskip

The following is the main result of the paper:
\begin{theorem}\label{thm:num-ind-2D}
   Let $M=(\{x,y,z\},d)$ a metric space of three points. Assume (without loss of generality) that $\dxy \geq \dxz \geq \dyz$. Then:
   \begin{itemize}
      \item if any triangular inequality on $M$ is satisfied as an equality, then ${n(\free{M})=1}$;
      \item if all triangular inequalities on $M$ are strict, then 
      \begin{equation*}        
         n(\free{M})= \max \bigg\{  \frac{\dxz \gro_y(x,z)}{\dxy \gro_z(x,y)+ \dyz\gro_x(y,z)  },  \frac{\dxz}{\dxy+\dyz}\bigg\},
      \end{equation*}
       where $\gro_z(x,y)$ is the Gromov product of $x$ and $y$ w.r.t. $z$.
   \end{itemize}
\end{theorem}

In particular, for a metric space of three points, the numerical index $1$ is only attained at the degenerated case---see Corollary \ref{cor:aligned}---and $n(\free{M})=\frac{1}{2}$ if and only if $M$ is an equilateral triangle---see Corollary \ref{cor:equil}. Moreover, simpler formulae are obtained for the isosceles case---see Corollaries \ref{cor:long-isos} and \ref{cor:fat-isos}---which allowed us to recover \cite[Theorem 1]{MarMer07} as a particular case.

Aside from this introduction, this article contains $4$ sections structured as follows. First, the remaining part of this section contains two subsections providing the necessary preliminaries and definitions concerning numerical index and Lipschitz-free Banach spaces. We encourage experts in the field to skip this part. Section \ref{sec:metric-tools} is fully devoted to developing some tools and results in metric spaces $M$ that will be needed to describe the geometric behavior of the numerical radius in $\free{M}$. In Section \ref{sec:num-rad-behaviour}, we begin the work in $2$-dimensional Lipschtiz-free spaces, getting the first results and estimations on the numerical radius of norm-one operators, deducing that the maximum appearing in Theorem \ref{thm:num-ind-2D} above works as a lower bound for the numerical index of $\free{M}$. In Section \ref{sec:const-op}, we will prove that the lower bound obtained in the previous section is indeed the exact formula to compute the numerical index of a $2$-dimensional $\free{M}$, by constructing operators that reach exactly this value as their numerical radius. Some applications, particular cases, and consequences on infinite dimensional Lipschitz-free spaces will be given in Section \ref{sec:app}. From here, we highlight the fact that the numerical index of a $2$-dimensional Lipschitz-free space can take any value of the interval $[\frac{1}{2},1]$.

\subsection*{Notation and basic Banach space notions}
Throughout this note, $X$ will denote a real Banach space, while its unit ball and sphere will be denoted by $B_{X}$ and $S_X$, respectively. Its dual space will be denoted by $X^*$, and given $x\in X$ and a linear continuous functional $x^*\in X^*$, we will write indistinctly $\la x,x^*\ra $ or $\la x^*,x\ra$ to refer to the evaluation $x^*(x)$. We will use $\mathcal{L}(X)$ to refer to the space of linear bounded operators from the Banach space $X$ into itself.

Given $X$ a real Banach space and $T\in \mathcal{L}(X)$, the \textit{numerical radius} of $T$ is defined as
\begin{equation*}
      v(T):=\sup\{ |\la x^*,Tx\ra| : x\in S_X, x^* \in S_{X^*}, \la x,x^*\ra=1 \}.
\end{equation*}
The numerical radius defines a seminorm in $\mathcal{L}(X)$ satisfying $v(T)\leq \|T\|$ for every $T \in \mathcal{L}(X)$. It may define an equivalent norm on $\mathcal{L}(\free{M})$, and this happens if and only if the equation 
\begin{equation*}
     C\|T\| \leq v(T)\leq \|T\|
\end{equation*}
is satisfied for a constant $0<C\leq 1$. 
  
The numerical index of $X$, denoted by $n(X)$, is the greatest constant $C\geq 0$ satisfying the equation above. Thus, if $X$ is a real Banach space, $n(X)$ belongs to $[0,1]$. The numerical index can equivalently be defined as
\begin{equation*}
      n(X)=\inf\{ v(T): T\in S_{\mathcal{L}(X)} \}.
\end{equation*}

As our working context is finite-dimensional Banach spaces, notice that it is enough to evaluate on the extreme points, i.e.,
\begin{equation}\label{eq:num-rad-finite}
      v(T):=\sup\{ |\la x^*,Tx\ra| : x\in \Ext{B_X}, x^* \in \Ext{B_{X^*}}, \la x,x^*\ra=1 \}.
\end{equation}

\subsection*{Lipschitz-free preliminaries}

For the sake of completeness, here we provide a very brief summary of some of the basic properties and tools needed to work with Lipschitz-free spaces. 
  
To introduce its definition, let $(M,d)$ be a metric space. By choosing a distinguished point $0_M\in M$ (usually called just $0$ when there is no possibility of confusion), we can take $\lip_0(M)$ as the linear space of Lipschitz functions $f\colon M\to \RR$ that vanish at $0$, which is a real Banach space when endowed with the norm of the Lipschitz constant, $\|\cdot\|_{\text{Lip}}$, defined by
\begin{equation*}
  \|f\|_{\text{Lip}}:= \sup_{x\neq y\in M} \dfrac{|f(x)-f(y)|}{d(x,y)}.
\end{equation*}

The choice of the element $0$ can be arbitrary since by taking $0'\neq 0$, $\lip_0(M)$ and $\lip_{0'}(M)$ are isometrically isomorphic as real Banach spaces. In order to simplify the notation, once the $0\in M$ is fixed, we will use the function $\pp(\cdot):= d(\cdot,0)$. 
The map $\dd: M \to \lip_0(M)^*$ that sends every point $x\in M$ to its evaluation functional $\dd_x$ is an isometry, and it is easy to see that $\dd_x$ and $\dd_y$ with $x \neq y$ are linearly independent. Notice that $\csp \{ \dd (x) : x\in M \}$ is a closed linear subspace of $\lip_0(M)^*$, and it is a predual of $\lip_0(M)$, denoted by $\free{M}$, i.e.,
\[ \free{M}= \csp\{\dd_x : x\in M\} \ \ ( \subset \lip_0(M)^* ).\]

Roughly speaking, we can think of $\free{M}$ as the Banach space constructed by taking $M$ and providing it with a linear structure in which distinct points in $M$ are now linearly independent, and the endowed norm is the one keeping the original metric structure of $M$. Thus, $\|\dd_x\|= \ppx$, or more generally, $\|\dd_x-\dd_y\|=\dxy$.
  
As the linear continuous functionals on $\free{M}$ are the elements of $\lip_0(M)$, it is clear how they act. Let $g\in \lip_0(M)$ and $\sum_i^n a_i \dd_{x_i}\in \free{M}$, then
\[\la g, \sum_i^n a_i \dd_{x_i} \ra= \sum_i^n a_i \la g,   \dd_{x_i} \ra= \sum_i^n a_i \dd_{g(x_i)}.\]
  
Although the space may be fully described using \emph{the deltas}, the following notion provides some advantages. Given $x,y \in M$ two distinct points in $M$, its associated \textbf{molecule} (some authors call this \textit{elemental molecule}) is
 \[m_{x,y}:= \frac{\dd_x -\dd_y}{d(x,y)}\in \free{M}.   \]
  
As, for every $x\in M$, ${\dd_x}={\ppx}\cdot m_{x,0}$, $\free{M}= \csp\{m_{x,y}: x,y \in M\}$. The reader will rapidly understand the relevance of these elements on $\free{M}$ since they are not only norm-one elements but, as $\la g, m_{x,y} \ra= \frac{g(x)-g(y)}{d(x,y)}$, they constitute a $1$-norming subset of $\free{M}$ for the norm $\|\cdot\|_{\lip}$. 
  
  Moreover, in the finite-dimensional case, all extreme points of the unit ball of $\free{M}$ are molecules (see, for instance, \cite{AlGui19}), although not every molecule should be an extreme point. Thus, even though the explicit computation of the norm of a finite-dimensional Lipschitz-free is far from being an easy task, one may represent its unit ball as the convex hull of all the molecules. By determining the extreme molecules we can also think of the faces of $B_{\free{M}}$ as convex hulls of some of these molecules.
  
  \begin{example}\label{ex:free-repres}
  \rm
  Consider $M=(\{x,y,0\},d)$. Then, $m_{x,0}$ and $m_{y,0}$ are norm-one vectors which are linearly independent, while 
  \begin{equation*}
      m_{x,y}=\frac{\ppx}{\dxy}m_{x,0} -\frac{\ppy}{\dxy}m_{y,0}.
  \end{equation*} 
  
  As Figure \ref{fig:free-repres} illustrates, the unit ball of the $2$-dimensional Lipschitz-free may be understood as a modification of the usual $\|\cdot\|_1$ norm in $\RR^2$. Indeed, we can interpret this as the norm in $\RR^2$ such that the unit ball is the absolute convex hull of the vectors $\{e_1,e_2, \frac{\ppx}{\dxy}e_1 -\frac{\ppy}{\dxy}e_2 \}$.
  \begin{figure}[htb!]
\begin{center}
\includegraphics[width=4in]{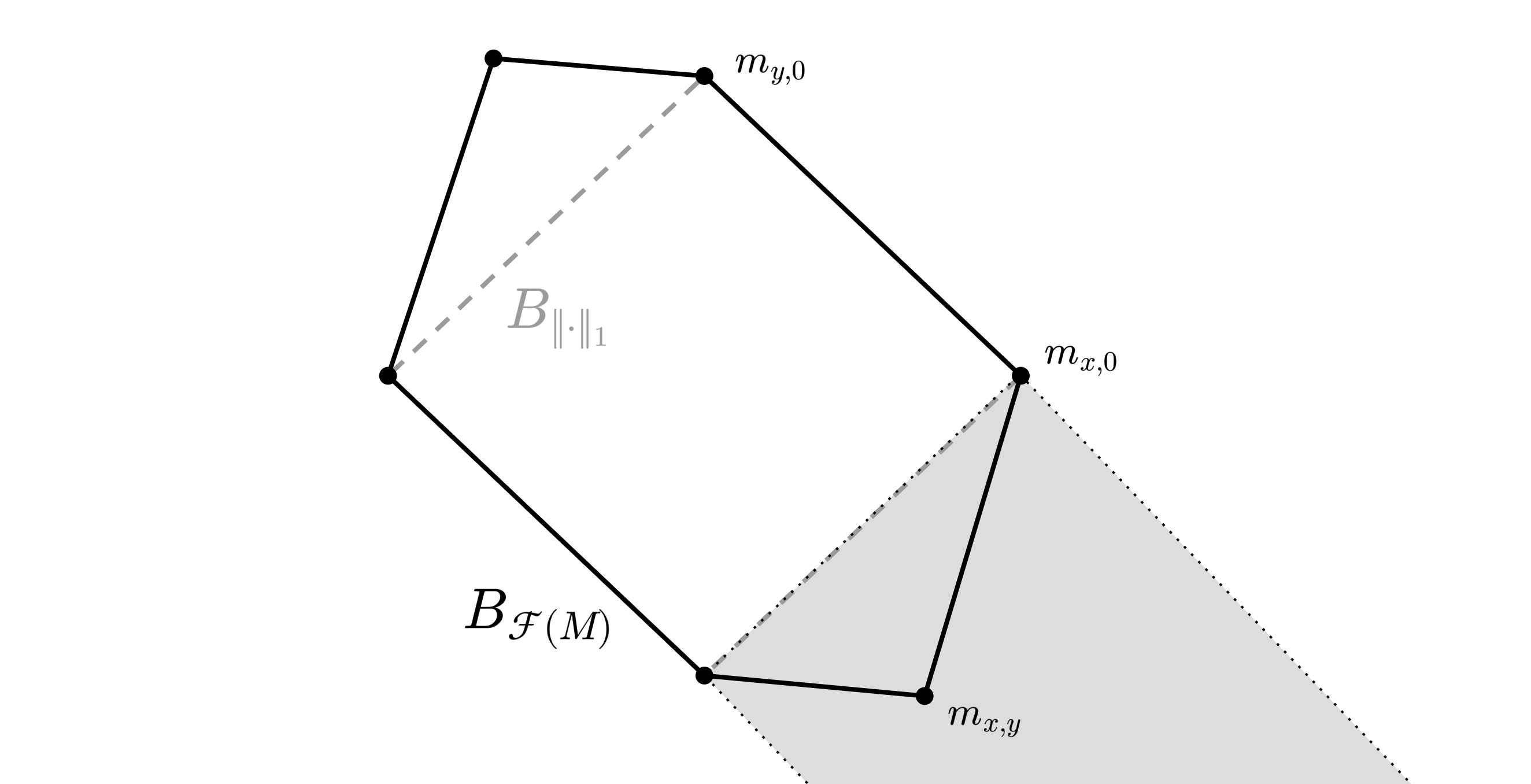}
\caption{Representation of $B_{\free{M}}$. The shaded area is where the molecule $m_{x,y}$ may be located, depending on $M$.}
\label{fig:free-repres}
\end{center}
\end{figure}

As Figure \ref{fig:free-repres} illustrates, given $M=(\{x,y,0\},d)$, the unit ball of $\free{M}$ will be represented as a specific hexagon, although depending on the location of the molecule $m_{x,y}$, it may be a rhombus. This happens exactly when $3$ molecules are \textit{aligned} (i.e. when a molecule is in the convex hull of the other two), and in the sections below we indicate a way to characterize this behaviour. When this happens, $\free{M}$ is isometrically isomorphic to ${(\RR^2, \|\cdot\|_1)}$.
\end{example}


\section{Metric Tools}\label{sec:metric-tools}
Along this section, we will define some constants on the metric space $M$ and state some of its properties that will be used later in order to work with the geometric structure in $\free{M}$ and to bring forth some geometric arguments to estimate the numerical index in $\free{M}$. To the best of our knowledge, some notions are introduced as a part of this work, while other ones, as easy computations on a metric space, may appear in other parts of the literature, maybe with other names. We will state its definitions and provide some properties for the sake of completeness. 
  
Before starting, it is worth remarking that, even though there is a need of having a distinguished point $0$ in $M$ to construct $\free{M}$, this element will play no special role in any part of the work. This is not only because the isometric structure of $\free{M}$ is independent of the choice of the distinguished element, but also because in any formula developed in this work, in any concept associated with a pair of points $x,y\in M$ or to a molecule $m_{x,y}\in \free{M}$, the computations are exactly the same either if one of the points $x$ or $y$ is the distinguished element in $M$ or not. To stress this fact, generally, we will put $M=(\{x,y,z\},d)$, without explicitly identifying the distinguished element. However, to make some deductions, sometimes we will need to assume (without loss of generality) any order in the distances. As the points in the metric space can be renamed as we please, we will assume the order ${\dxy \geq \dxz \geq \dyz }$, and just in this scenario, as a way to ease the notation, we will put $z=0$. In this case, we will write $M=(\{x,y,0\},d)$ and the order will be $\dxy \geq \ppx \geq \ppy$, where $\pp(\cdot):= d(\cdot,0)$.

Let $(M,d)$ be a metric space and $x,y,z\in M$ a triplet of points. We say that $z$ \textit{lies in the metric segment between} $x$ and $y$ if $d(x,y)= d(x,z)+ d(z,y)$. If there is a triplet of points $x,y,z\in M$ such that one of them lies in the metric segment of the other two, we may just say that the three points are \textit{(metrically) aligned}. The following notion is well known and works as a tool to detect this behavior in the metric space $M$.
 \begin{definition}
 Let $(M,d)$ be a metric space. Given distinct $x$, $y$  and $z\in M$, the \textbf{Gromov product} of $x$ and $y$ with respect to $z$ is 
 \begin{equation*}
     \gro_z(x,y):= d(x,z)+d(y,z)-d(x,y).
 \end{equation*}
 \end{definition}
 
This expression comes from the triangular inequality $d(x,z)+d(y,z) \geq d(x,y)$, so Gromov products are always non-negative. Roughly speaking, the Gromov product of $x$ and $y$ with respect to $z$ measures how far is $z$ from being in the metric segment between $x$ and $y$, which happens if and only if $ \gro_z(x,y)=0$. This allows us to provide a general definition of being a \textit{non-aligned} metric space.
 
 \begin{definition}\label{def:non-aligned}
 A given metric space $(M,d)$ will be called \textbf{non-aligned} if every Gromov product is strictly positive. A non-aligned metric space $M$ with three elements will be called \textbf{triangle} through this paper. 
 \end{definition}

The following lemma is straightforward but will be of use.
\begin{lemma}\label{lemma:grom-order}
Let $M=(\{x,y,z\},d)$ be a triangle. Assume that $d(x,y)\geq d(x,z)\geq d(y,z)$. Then $\gro_z(x,y)\leq \gro_y(x,z)\leq \gro_x(y,z)$.
\end{lemma}
%

We introduce a related notion that will be useful to simplify computations.

\begin{definition}
Let $(M,d)$ be a metric space. Given distinct $x$, $y$  and $z\in M$, the \textbf{weighted Gromov product} of $x$ and $y$ with respect to $z$ is defined as
\begin{equation*}
    \wgro_z(x,y):= d(x,y)\gro_z(x,y).
\end{equation*}
\end{definition}

 
\begin{lemma}\label{lemma:wgro-order}
Let $M=(\{x,y,z\},d)$ be a triangle. Then, $\wgro_z(x,y)\leq \wgro_y(x,z)$ if and only if $d(x,y)\geq d(x,z)$.
\end{lemma}
\begin{proof} It is routine by expanding expressions and factorizing properly.
%
\end{proof}

We now introduce the two quantities---and some of their properties---that will be essential for expressing the numerical index.

\begin{definition} \label{def:contop}
Let $M=(\{x,y,z\},d)$ be a triangle. 
Given a pair of points $x,y\in M$, we call its \textbf{optimal contribution} to 
\begin{equation*}
    \contopxy:= \dfrac{d(x,y) \gro_z(x,y)}{ \dxz \gro_y(x,z) +  \dyz \gro_x(y,z)} =\dfrac{\wgro_z(x,y)}{\wgro_y(x,z)+\wgro_x(y,z)}.
\end{equation*}
\end{definition}


\begin{lemma}\label{lemma:contop-order}
Let $M=(\{x,y,z\},d)$ be a triangle. Then, $d(x,y)\geq d(x,z)$  if and only if $\contopxy\leq \contopxz$.
\end{lemma}
\begin{proof}
It is routine by expanding expressions, factorizing properly, and using Lemma \ref{lemma:wgro-order}.
 %
 %
\end{proof}

\begin{definition}
Let $(M,d)$ be a metric space. Then, given a triplet of points $x,y,z\in M$, we call its \textbf{metric ratio} to 
\begin{equation*}
    \rxyz:= \frac{\dxy}{d(x,z)+ d(y,z)}.
\end{equation*}
\end{definition}


\begin{lemma}\label{lemma:ratios-order}
Let $M=(\{x,y,z\},d)$ a metric space with 3 points. Then, $\rxyz\geq \rxz $ if and only if $d(x,y)\geq d(x,z)$.
\end{lemma}


We establish the relationship between optimal contribution and metric ratio.

\begin{lemma} \label{lemma:ratio-vs-contop}
Let $M=(\{x,y,0\},d)$ be a triangle with ${\dxy \geq \ppx \geq \ppy}$. Then, $\rx\geq \contopx$ if and only if $\ppx \geq \dfrac{\dxy^2 + \ppy^2}{\dxy +\ppy}$.
\end{lemma}
\begin{proof}
It is routine by expanding the expressions involved and factorizing properly.
\end{proof}

\begin{lemma}\label{lemma:order-on-extremes}
Let $M=(\{x,y,0\},d)$ be a triangle with ${\dxy \geq \ppx \geq \ppy}$. Then, 
\begin{align*}
    &\max\{\contopxy, \rxy\}= \rxy;\\
    &\max\{\contopy, \ry\}= \contopy.
\end{align*}

Moreover, $\rxy \geq \contopy$ if and only if $\ppx \leq \dfrac{\dxy^2 + \ppy^2}{\dxy +\ppy}$.
\end{lemma}
\begin{proof}
First, we will prove that for the longest side of the triangle ($d(x,y)$), its ratio is greater than its optimal contribution. Indeed, by Lemma \ref{lemma:grom-order}, we know that $\gro_z(x,y)\leq \gro_y(x,z)\leq \gro_x(y,z)$, which implies that
\begin{align*}
    \contop(m_{x,y}) &= \dfrac{d(x,y) \gro_z(x,y)}{ d(x,z)\gro_y(x,z) +  d(y,z) \gro_x(y,z)}\\ 
    &\leq  \dfrac{d(x,y) \gro_y(x,z)}{ d(x,z)\gro_y(x,z) +  d(y,z) \gro_y(x,z)}\\ 
    &= \dfrac{d(x,y)}{d(x,z)+d(y,z)}= \rxyz.
\end{align*}
 
Analogously, for the shortest side ($d(y,z)$), the inequality is the opposite. Indeed, 
\begin{align*}
    \contop(m_{y,z}) &= \dfrac{d(y,z) \gro_z(y,z)}{ d(x,y)\gro_z(x,y) +  d(x,z) \gro_y(x,z)}\\ 
    &\geq  \dfrac{d(y,z) \gro_y(x,z)}{ d(x,y)\gro_y(x,z) +  d(x,z) \gro_y(x,z)}\\
    &= \dfrac{d(y,z)}{d(x,y)+d(x,z)}= \ryz.
\end{align*} 

Moreover, developing the inequality  $\rxy \geq \contopy$, one obtains the equivalent inequality
\[
\dxy(\wgroxy + \wgrox) \geq \wgroy (\ppx + \ppy).
\]
 
 
 Now, we make a full expansion of the expression and rearrange the result, writing it as a polynomial of variable $\ppx$, reaching %
 a concave parabola with roots
 \[\ppx= \dxy - \ppy \quad \text{and} \quad \ppx= \dfrac{\dxy^2 + \ppy^2}{\dxy +\ppy}.\]
Therefore, $\rxy \geq \contopy$ if and only if 
 \[
     \dxy - \ppy \leq \ppx \leq \dfrac{\dxy^2 + \ppy^2}{\dxy +\ppy}.
 \]

The left-hand side of the expression above is always satisfied as a consequence of triangular inequality, so we conclude that $\rxy \geq \contopy$ if and only if $\ppx \leq \frac{\dxy^2 + \ppy^2}{\dxy +\ppy}$, as we wanted to prove. 
\end{proof}


\section{Geometric behaviour of the Lipschitz-free space associated to a triangle}\label{sec:num-rad-behaviour}

In this section, we will begin with the assessment of the numerical index in $\free{M}$, where $M$ is a metric space of 3 points. The whole section is devoted to proving that $n(\free{M})$ is greater than or equal to the estimation given in Theorem \ref{thm:num-ind-2D}.%
We introduce first some notions of an auxiliary character.

\begin{definition}
  Let $X$ be a Banach space, $x\in \Ext{B_X}$, and $T\in \mathcal{L}(X)$. We refer as the \textit{contribution of} $x$ \textit{to the numerical radius of } $T$ to
  \begin{equation*}
      \nu(T,x):= \sup\{|\la x^*, Tx \ra|: x^*\in \Ext{B_{X^*}}, \la x,x^* \ra=1\}.
  \end{equation*}
\end{definition}
  
Given $X$ a Banach space and $x^*\in S_{X^*}$ a norm-attaining element, we say that the set
\[
        \face(x^*):= \{x\in S_X:\la x,x^* \ra=1\}
\]
is the \textit{face} generated by $x^*$ in $X$.

Now, considering $M$ a metric space of three points, generally, it is expected a hexagonal shape on the unit ball---recall Figure \ref{fig:free-repres}. Then, it is useful to identify the linear forms defining its maximal faces, as they are the only ones that we need to evaluate the numerical radius of an operator. 

Given $M=(\{x,y,z\},d)$ a metric space of three points, for every $z\in M$, consider $\pp_z \in {\free{M}^*}$ the norm-one linear form characterized by
\[{\la \pp_z, m_{x,z} \ra=1=\la \pp_z, m_{y,z} \ra}.\]
By the very definition of $m_{x,y}$, it follows that $\la \pp_z, m_{x,y}\ra = \frac{\dxz-\dyz}{\dxy}$. 

Therefore, we have that the six maximal faces of $B_{\free{M}}$ are $\face(\pm \pp_x)$, $\face(\pm \pp_y)$ and  $\face(\pm \pp_z)$.%
We will use $\pp_z$ to denote this linear continuous functionals acting on $\free{M}$, while $\pp$ will be used, as introduced before, as the ``distance to zero'' Lipschitz map acting on $M$.

\begin{figure}[h!]
\begin{center}
\includegraphics[width=4in]{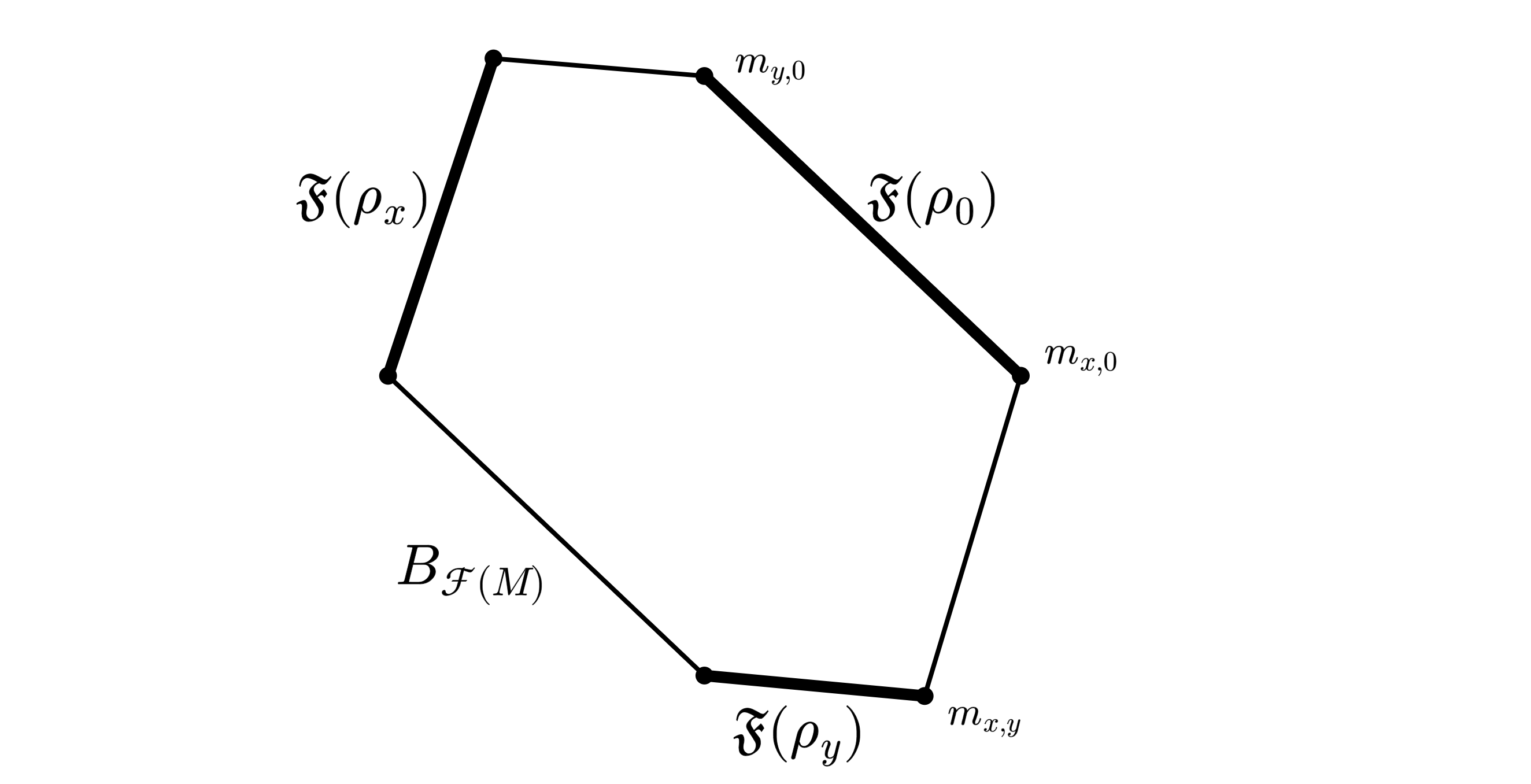}
\caption{Maximal faces in $B_{\free{M}}$.}
\label{fig:free-faces}
\end{center}
\end{figure}

However, depending on the underlying metric structure of $M$, this hexagon can also be a rhombus. This happens---as Figure \ref{fig:free-repres} illustrates---if and only if one molecule is in the convex hull of the other two, which is determined by the underlying metric structure of $M$, as a molecule $m_{x,y}$ lies in $\text{conv}\{m_{x,z}, m_{z,y}\}$ if and only if $d(x,y)= d(x,z)+ d(z,y)$,---i.e., $z$ lies in the metric segment between $x$ and $y$ or, equivalently, $G_z(x,y)=0$. To summarize, $B_{\free{M}}$ will be a rhombus---being $\free{M}$ isometrically isomorphic to $(\RR^2,\|\cdot\|_1)$---if and only if $M$ is not a triangle. It is well known that in this case $n(\free{M})=1$, covering the first part of Theorem \ref{thm:num-ind-2D}. Therefore, from this moment on we will avoid this trivial case by considering $M$ to be a triangle. 

\subsection{The first lower bound: the optimal contribution}

 We introduced the optimal contribution---see Definition \ref{def:contop}---as this quotient appears naturally as a lower bound for the numerical radius of a norm-one operator that sends a molecule $m_{x,y}$ to the sphere, as this is the lowest value of $\nu(T,m_{x,y})$ that such an operator can have. This phenomenon is reflected in the following result. 
\begin{lemma}\label{lemma:contop-bound}
Let $M=(\{x,y,z\},d)$ be a triangle, and $T\in S_{\free{M}}$ be such that $\|T m_{x,y}\|=1$. Then, 
\[\nu(T,m_{x,y})\geq     \contopxy.\]

Moreover, if 
\[T m_{x,y}= \lambda_z m_{x,z} + (1-\lambda_z)m_{y,z},\]
where
\[\lambda_z  := \dfrac{\wgro_y(x,z)}{\wgro_y(x,z)+\wgro_x(y,z)},\]
then $\nu(T,m_{x,y})=\contopxy$.

\end{lemma}
\begin{proof}
If $Tm_{x,y}\in \face(\pm \pp_x)$ or $Tm_{x,y}\in \face(\pm \pp_y)$, then $\nu(T,m_{x,y})=1$. Thus, we consider now, without loss of generality, that $Tm_{x,y}\in \face(\pp_z)$ (otherwise, we take the operator $-T$). Then, the operator must be of the form

\[ T_\lambda m_{x,y}= \lambda m_{x,z} + (1-\lambda)m_{y,z}, \quad \text{with } \lambda\in [0,1].\]

Therefore, the evaluations for the contribution of $m_{x,y}$ are 
\begin{align*}
&\la \pp_x , T_\lambda m_{x,y}\ra= -\lambda + (1-\lambda) \dfrac{d(x,y)-\dxz}{\dyz},\\
&\la \pp_y , T_\lambda m_{x,y}\ra= \lambda \dfrac{d(x,y)-\dyz}{\dxz} - (1-\lambda).
\end{align*}

A simple argument shows that considering $|\la \pp_x , T_\lambda m_{x,y}\ra|$ and $ |   \la \pp_y , T_\lambda m_{x,y}\ra|$, for any fixed values of the distances, one of them is increasing with the value of $\lambda$, while the other one is decreasing. Thus, the contribution of $m_{x,y}$, $\nu(T_\lambda, m_{x,y})= \max\{|\la \pp_x , T_\lambda m_{x,y}\ra|, |   \la \pp_y , T_\lambda m_{x,y}\ra|\}$, is minimized for the value of $\lambda$ satisfying $|\la \pp_x , T_\lambda m_{x,y}\ra|=|\la \pp_y , T_\lambda m_{x,y}\ra|$.

Then, the best value of $\lambda$ is the one that solves the equation
\[ \la \pp_x , T_\lambda m_{x,y}\ra=\la \pp_y , T_\lambda m_{x,y}\ra.\]
since, otherwise, if $\la \pp_x , T_\lambda m_{x,y}\ra=-\la \pp_y , T_\lambda m_{x,y}\ra  $, then $\lambda \in [0,1]$ if and only if $-\dyz G_z(x,y)\geq 0$, which is impossible. It is straightforward then that this optimal value of $\lambda$ is
\[\lambda_z := \dfrac{\dxz\gro_y(x,z)}{ \dxz\gro_y(x,z) +  \dyz\gro_x(y,z)} = \dfrac{\wgro_y(x,z)}{\wgro_y(x,z)+\wgro_x(y,z)}.
\]

Recall that both the numerator and denominator are strictly positive as a consequence of $M$ being a triangle and that---as can be checked---$\lambda_z$ belongs to $[0,1]$. By conveniently write
\[1-\lambda_z  = \dfrac{\wgro_x(y,z)}{\wgro_y(x,z)+\wgro_x(y,z)}.\]
and expand, we obtain that
\begin{align*}
\nu(T_{\lambda_z}, m_{x,y})&= \dfrac{d(x,y) \gro_z(x,y)}{ \dxz\gro_y(x,z) +  \dyz \gro_x(y,z)}\\ 
&=\dfrac{\wgro_z(x,y)}{\wgro_y(x,z)+\wgro_x(y,z)} = \contopxy.   
\end{align*}

Thus, as any operator $T$ with $Tm_{x,y}\in \face(\pp_z)$ is of the form $T_\lambda$, we would have $\nu(T_\lambda, m_{x,y})\geq \nu(T_{\lambda_z}, m_{x,y}) = \contopxy$.
\end{proof}
  
It is worth noticing that if we pick a molecule $m_{x,y}$, we can construct a triangle $M$  such that $\contopxy$ may be as close to zero as we please. Still, it is not possible to achieve an operator with numerical index arbitrarily small. Roughly speaking, decreasing the value of an optimal contribution will lead to an increase in the contributions of the other molecules. In this case, the optimal contribution will fail to accurately estimate the numerical radius. There is another value working as a lower bound for the numerical radius of an operator, which takes its place as the best estimation of the numerical index when the optimal contribution fails.

Before continuing with the argument, it seems remarkable to us that the optimal contribution is enough to reach a relevant insight.

\begin{proposition}
\label{prop:2D-1/2}
Let $M=(\{x,y,z\},d)$ be a metric space. Then, the numerical radius of $\free{M}$ is greater than or equal to $\frac{1}{2}$. Moreover, equality implies $M$ to be equilateral.
\end{proposition}

\begin{proof}
First, by \emph{reductio ad absurdum}, assume that there exists a norm one operator $T\in {\mathcal{L}(\free{M})}$ with $\nu(T)<\frac{1}{2}$. The metric space $M$ must be non-aligned since, if the points of $M$ are aligned, then ${n(\free{M})=1}$, which leads to a contradiction. Once again, $\|T\|=1$, implies that $T(B_{\free{M}})\subset B_{\free{M}}$ and that $Tm\in S_{\free{M}}$ for some $m\in \{m_{x,z}, m_{y,z}, m_{x,y}\}$. Assume without loss of generality that $Tm_{x,y}\in S_{\free{M}}$. As $\nu(T)<1$, either $Tm_{x,y}\in \face(\pp_z)$ or $Tm_{x,y}\in \face(-\pp_z)$. Now, $\nu(T)<\frac{1}{2}$ implies that $|\la \pp_z,Tm_{x,z} \ra|<\frac{1}{2}$ and $|\la \pp_z,Tm_{y,z} \ra|<\frac{1}{2}$. By the linearity of $T$, and by using that
\[
    Tm_{x,y}= \frac{\dxz}{d(x,y)} Tm_{x,z}-  \frac{\dyz}{d(x,y)} Tm_{y,z},
\]
we can estimate
\begin{align*}
|\la \pp_z, Tm_{x,y}\ra|&= \bigg|\frac{\dxz}{d(x,y)} \la\pp_z,Tm_{x,z}\ra-  \frac{\dyz}{d(x,y)} \la \pp_z, Tm_{y,z}\ra \bigg|\\
        &\leq \frac{\dxz}{d(x,y)} |\la\pp_z,Tm_{x,z}\ra|+  \frac{\dyz}{d(x,y)} |\la \pp_z, Tm_{y,z}\ra|  \\
        &< \frac{1}{2}\bigg( \frac{\dxz}{d(x,y)} +  \frac{\dyz}{d(x,y)} \bigg)= \dfrac{\dxz+\dyz}{2 d(x,y)}.\\    
\end{align*}
As $Tm_{x,y}\in \face(\pp_z)$ or $Tm_{x,y}\in \face(-\pp_z)$, we deduce that
\[  1= \|Tm_{x,y}\|= |\la \pp_z, Tm_{x,y}\ra|<\dfrac{\dxz+\dyz}{2 d(x,y)},\]

which implies that 
\begin{equation}\label{eq_ine1}
\dxz+\dyz>2d(x,y).\end{equation} However, $\nu(T)<\frac{1}{2}$ also implies that the contribution of $m_{x,y}$ must be strictly lower than $\frac{1}{2}$, but, at the same time, greater than the optimal contribution, i.e.,
\begin{equation*}
    \frac{1}{2}> \nu(T,m_{x,y})\geq \contopxy
\end{equation*}

This would imply that 
\begin{equation*}
    \dxz\gro_y(x,z)+ \dyz\gro_x(y,z)> 2d(x,y)G_z(x,y).
\end{equation*}
which, by expanding and factorizing properly, is equivalent to 
\begin{equation*}
  -(\dxz-\dyz )^2 > d(x,y)(\dxz+\dyz-2d(x,y)).
\end{equation*}

By \eqref{eq_ine1}, the right-hand side is strictly positive, which leads to a contradiction. Then, no triangle $M$ has access to an operator $T$ with $\nu(T)<\frac{1}{2}$. Moreover, if the whole argument above is repeated with $\nu(T)\leq \frac{1}{2}$---we let the details to the reader---then, it is deduced that $\dxz=\dyz=d(x,y)$, i.e., $M$ is equilateral.
\end{proof}

\begin{remark}
    \rm We will see in Section~\ref{sec:app} that equality in Proposition \ref{prop:2D-1/2} equivales to $M$ being equilateral. \eor
\end{remark}

\subsection{The second lower bound: the metric ratio}

\begin{proposition}
Let $M=(\{x,y,z\},d)$ be a triangle and $T\in S_{\mathcal{L}(\free{M})}$ such that $\|Tm_{x,y}\|=1$. Then, $\nu(T)\geq \rxyz$.
\end{proposition}
\begin{proof}
 By \emph{reductio ad absurdum}, assume that $\nu(T)<\rxyz$. Then, by the triangular inequality, $Tm_{x,y} \in \face(\pp_z)$ or $Tm_{x,y} \in \face(-\pp_z)$ since $\rxyz \leq 1$. Without loss of generality, we assume $Tm_{x,y} \in \face(\pp_z)$---that is, $\la \pp_z,Tm_{x,y}\ra =1$. By the linearity of $T$, it holds that
\begin{equation}\label{eq:linearity-mxy}
    Tm_{x,y}= \frac{\dxz}{\dxy} Tm_{x,z}- \frac{\dyz}{\dxy} Tm_{y,z},
\end{equation}
and evaluating against $\pp_z$, we deduce that
\begin{equation} \label{eq:rho-ratio}
 1= \la \pp_z, Tm_{x,y}\ra= \frac{\dxz}{\dxy} \la \pp_z, Tm_{x,z}\ra- \frac{\dyz}{\dxy} \la \pp_z, Tm_{y,z}\ra.
\end{equation}

Now, there are two options, either $\la \pp_z, Tm_{x,z}\ra\leq 0$ or $\la \pp_z, Tm_{x,z}\ra> 0$. We will see that both lead to a contradiction. Indeed, if  $\la \pp_z, Tm_{x,z}\ra\leq 0$, by equation \eqref{eq:rho-ratio}, we deduce
\begin{equation*}
    \begin{split}
        |\la \pp_z, Tm_{y,z}\ra| &= \frac{\dxy}{\dyz} \Big| \frac{\dxz}{\dxy} \la \pp_z, Tm_{x,z}\ra-1  \Big| \\
       &\geq \frac{\dxy}{\dyz} \geq \frac{\dxy}{\dxz + \dyz} = \rxyz,
    \end{split}
\end{equation*}
which implies that $\nu(T)\geq \nu(T,m_{y,z})\geq  |\la \pp_z, Tm_{y,z}\ra| \geq \rxyz$, and yields a contradiction. On the other hand, if $\la \pp_z, Tm_{x,z}\ra$ is positive we would have that  
\begin{equation*}
    0<\la \pp_z, Tm_{x,z}\ra \leq \nu(T)<\rxyz ,
\end{equation*}
which implies that
\begin{equation*}
    0<\frac{\dxz}{\dxy} \la \pp_z, Tm_{x,z}\ra < \frac{\dxz}{\dxy}\rxyz = \frac{\dxz}{\dxz+\dyz}<1.
\end{equation*}

Therefore, by equation \eqref{eq:rho-ratio}, it implies that ${\la\pp_z, Tm_{y,z} \ra<0}$. But then, 
\begin{equation*}
\begin{split}
    \la\pp_z, Tm_{y,z} \ra &= \frac{\dxy}{\dyz} \Big( \frac{\dxz}{\dxy} \la \pp_z, Tm_{x,z}\ra-1  \Big)\\
   & < \frac{\dxy}{\dyz} \Big( \frac{\dxz}{\dxy} \rxyz-1  \Big)\\
  & = \frac{\dxz}{\dyz} \rxyz - \frac{\dxy}{\dyz}
=-\rxyz,
    \end{split}
\end{equation*}

which leads again to the contradiction 
\[\nu(T)\geq |  \la\pp_z, Tm_{y,z} \ra | = -   \la\pp_z, Tm_{y,z} \ra >\rxyz.\qedhere \] 
\end{proof}

We have just proved that if a norm-one operator sends a specific extreme point to the sphere, two values work as lower bounds for the numerical radius, so we may take the greatest of them to control the numerical radius. We state this for further reference.

\begin{corollary}\label{cor:radius-bound}
Let $M=(\{x,y,z\},d)$  be a triangle and $T\in S_{\mathcal{L}(\free{M})}$ satisfying $\|Tm_{x,y}\|=1$. Then, 
\begin{equation*}
    \nu(T)\geq \max\{\contopxy, \rxyz\}.
\end{equation*}
\end{corollary}

As a norm-one operator in a finite-dimensional $\free{M}$ must send one of the extreme molecules to the sphere, the following provides a first lower bound for the numerical index.

\begin{proposition}\label{prop:num-ind-bound}
Let $M=(\{x,y,z\},d)$ a triangle. Then
\begin{align*}
    n(\free{M})\geq    \min\Big\{&\max\{\contop(m_{x,y}), \rxyz\},\\
    &\max\{\contop(m_{x,z}), \rxz\},\\
    &\max\{\contop(m_{y,z}), \ryz\} \Big\}
\end{align*}
\end{proposition}
\begin{proof}
Fix any $T\in S_{\mathcal{L}(\free{M})}$. Then, $T(B_{\free{M}})\subset B_{\free{M}}$ and, by the convexity of the unit ball and linearity of $T$, we deduce that $Tm\in S_{\free{M}}$ for some $m\in \{m_{x,z}, m_{y,z}, m_{x,y}\}$. Depending on which molecule is $T$ sending to the sphere, we may take its corresponding lower bound of the numerical radius given by Corollary \ref{cor:radius-bound}. As the numerical index minimizes the numerical radius, we will consider the best-case scenario: taking the minimum of all three lower bounds. 
\end{proof}

\textbf{Remark:} At this point, the previous Proposition already hints that depending on the computation of that minimum, we know which molecule an operator $T$ must send to the sphere to attain the numerical radius. So we could expect this case to provide the operator with the smallest numerical radius. However, the last result can be dramatically improved. Once the ordering in the distances on the metric space $M$ is known, all the relations between the metric ratios and the optimal contributions are clear, so there is no need to compute the whole minimum. Thus, we are in conditions to provide a simpler estimation of the numerical index of $\free{M}$ by refining---by means of the results of Section~\ref{sec:metric-tools}---Proposition \ref{prop:num-ind-bound}.

\begin{proposition} \label{prop:num-ind-optimal-bound}
Let $M=(\{x,y,0\},d)$ be a triangle. Assume that it satisfies ${\dxy \geq \ppx \geq \ppy}$. Then,
\begin{equation*}
    n(\free{M})\geq     \max\{\contopx, \rx\}.
\end{equation*}
\end{proposition}
\begin{proof}
Through the assumption of the ordering of the distances and Lemma \ref{lemma:order-on-extremes}, the maximums at left and right of the inequality obtained in Proposition~\ref{prop:num-ind-bound} are determined, so we can deduce that 
\begin{equation*}
    n(\free{M})\geq    \min\Big\{  \rxy, \max\{\contopx, \rx\}, \contopy \Big\}.
\end{equation*}

At this point, we find a dichotomy in which always prevails the maximum in the middle. Indeed, there are two cases.

In one hand, if $\ppx\geq  \dfrac{\dxy^2 + \ppy^2}{\dxy +\ppy}$, by Lemma \ref{lemma:order-on-extremes}, $\rxy \leq \contopy$ and, by Lemma \ref{lemma:ratio-vs-contop}, $\max\{\contopx, \rx\}= \rx$. Then, by the order in the ratios---see Lemma \ref{lemma:ratios-order}---we conclude that $\rx$ is lower.

On the other hand, if $\ppx\leq  \dfrac{\dxy^2 + \ppy^2}{\dxy +\ppy}$, Lemma \ref{lemma:order-on-extremes} states that $\rxy \geq \contopy$ and, by Lemma \ref{lemma:ratio-vs-contop}, $\max\{\contopx, \rx\}= \contopx$. Then, by the order in the optimal contributions---see Lemma \ref{lemma:contop-order}---we know that the minimum value is $\contopx$.

Combining both cases, we have
\begin{equation*}
    n(\free{M})\geq    \max\{\contopx, \rx\}.
\end{equation*}
The proof is over.
\end{proof}

\section{Attaining the numerical index: construction of the operators}
\label{sec:const-op}

In the previous section, we found a lower bound for the numerical index of $\free{M}$ when $M$ is a triangle. This section is devoted to showing that this bound is sharp. We must construct norm-one operators in $\free{M}$ attaining this value as its numerical radius. Thus, given a triangle $M=(\{x,y,0\},d)$ with $\dxy \geq \ppx \geq \ppy$, we want to see that $n(\free{M})= \max\{\contopx, \rx\}$. We will refer to such a triangle $M$ as \textit{canonical}.

The maximum between these two values depends on a very geometrical condition on the triangle $M$---which is again translated to the geometry of $B_{\free{M}}$---and depending on which one is the greatest of both quantities, we will have two different scenarios, each one requiring a different construction of an operator with the lowest numerical radius. 

Therefore, we split the argument in two cases, Subsection \ref{subsec:op-ratio} holds the construction of the operator when $\rx = \max\{\contopx, \rx\}$, while Subsection \ref{subsec:op-contop} contains the case $\contopx= \max\{\contopx, \rx\}$. Notice that the explicit constructions are in Propositions \ref{prop:constr-op-ratio} and \ref{prop:constr-op-contop}, while the other results are just auxiliary lemmas.
\subsection{Operator for the metric ratio}\label{subsec:op-ratio}

\begin{lemma} \label{lemma:bound-mx0-ratio} 
Let $M=(\{x,y,0\},d)$ be a canonical triangle 
and $\rx= \max\{\contopx, \rx \}$. Then, the vector
\[
    u:= \aa_x m_{y,x}+ (1-\aa_x)m_{y,0} \in S_{\free{M}},
\]
where 
\[
    \aa_x := 1- \dfrac{\ppy \grox}{(\dxy +\ppy) \groxy},
\]

satisfies $|\la \pp_0, u\ra|\leq \rx$.
\end{lemma}
\begin{proof}
We will split the proof into two parts. First, 
\begin{align*}
1-\aa_x \frac{\groy}{\dxy} \leq \rx&\iff \frac{\grox}{\groy} \frac{\dxy}{\dxy+\ppy} \leq \aa_x 
 \end{align*}
 After the expansion of this expression, we reach the equivalent one
 \begin{equation*}
     \begin{split}
          \dxy\groxy \Big(\grox - \groy\Big) + \ppy\groy \Big(\grox-\groxy\Big)&\leq 0,
     \end{split}
 \end{equation*}
 and further simplifications lead to the inequality
 \begin{align*}
       \ppx \Big(\dxy^2 + \ppy^2\Big) \leq \ppx^2\Big(\dxy +\ppy\Big), 
\end{align*}
which is again equivalent to 
\begin{align*}
     \dfrac{\dxy^2 + \ppy^2}{\dxy +\ppy} \leq \ppx.
\end{align*}
 
 Thus, the condition $\la \pp_0, u\ra\leq \rx$ is equivalent to this last inequation, and by Lemma \ref{lemma:ratio-vs-contop}, we know that this is once again equivalent to the hypothesis of $\rx= \max\{\contopx, \rx \}$.

Now, we will check the other inequality. Indeed,
\begin{align*}
      1-\aa_x \frac{\groy}{\dxy} \geq -\rx
    \iff   1+ \rx  \geq \aa_x \frac{\groy}{\dxy}.
\end{align*}

Through a full expansion, analogous to the previous one, we reach the expression
\begin{align*}
        2\ppx\ppy\dxy &\geq 2\ppx^2\ppy - 2 \ppx \ppy^2, 
\end{align*}
which is equivalent to $\dxy \geq  \ppx - \ppy$. But this is simply the triangular inequality and, therefore, true.

Finally, combining both inequalities, we reach $|\la \pp_0, u\ra|\leq \rx$.
\end{proof}

Here we just state the computations of the optimal case provided in Lemma \ref{lemma:contop-bound}, but for the molecule $m_{y,0}$.

\begin{lemma}\label{lemma:opt-y}
Let $M=(\{x,y,0\},d)$ be a canonical triangle.
Let $\lambda_y$ be the optimal $\lambda$ for $m_{y,0}$ (the analogous provided by Lemma \ref{lemma:contop-bound}), and take $v:= \lambda_y m_{y,x}- (1-\lambda_y) m_{x,0}$. Then, 
\begin{align*}
    \la \pp_0,v \ra&= -1 + \lambda_y \dfrac{\grox}{\dxy}= -\contopy;\\
      \la \pp_y,v \ra&= - \Big(1 - (1- \lambda_y) \dfrac{\groxy}{\ppx}\Big)= 
      -\contopy;\\
         \la \pp_x,v \ra&=1.
\end{align*}
\end{lemma}

\begin{proposition}\label{prop:constr-op-ratio}
Let $M=(\{x,y,0\},d)$ be a canonical triangle
and \[\rx= \max\{\contopx, \rx.\}\] Then, there exists $T \in S_{\mathcal{L}(\free{M})}$ such that $\nu(T)= \rx$.
\end{proposition}

\begin{proof}
To construct a linear operator, it is enough to define the image of $m_{x,0}$ and $m_{y,0}$. First, consider $T$ a linear operator such that
\begin{equation*}
    Tm_{x,0}= \aa_x m_{y,x}+ (1-\aa_x)m_{y,0} \in S_{\free{M}}
\end{equation*}
where 
\begin{equation*}
    \aa_x := 1- \dfrac{\ppy \grox}{(\dxy +\ppy) \groxy}.
\end{equation*} 


Then, we can check that
\begin{equation*}
\begin{split}
    \la \pp_0,T m_{x,0} \ra&= 1 - \aa_x \dfrac{\groy}{\dxy};\\
      \la \pp_x,T m_{x,0} \ra&= 1 - (1- \aa_x) \dfrac{\groxy}{\ppy} = \rx;\\
         \la \pp_y,T m_{x,0} \ra&=-1.
    \end{split}
\end{equation*}

By Lemma \ref{lemma:bound-mx0-ratio}, we know that $ \la \pp_0,T m_{x,0} \ra \leq \rx$, so $\nu(T,m_{x,0})=|  \la \pp_x,T m_{x,0} \ra|= \rx$. 

Now, consider $T$ an operator such that

\begin{equation*}
    Tm_{y,0}= \dfrac{\rx}{\contopy} (\lambda_y m_{y,x}- (1-\lambda_y)m_{x,0}),
\end{equation*}

where $\lambda_y$ is the one in Lemma \ref{lemma:opt-y}. As a consequence of this result, it is easy to see that such an operator satisfies
\begin{equation*}
\begin{split}
    \la \pp_0,T m_{y,0} \ra&= -\rx;\\
      \la \pp_y,T m_{y,0} \ra&= -\rx;\\
         \la \pp_x,T m_{y,0} \ra&=\dfrac{\rx}{\contopy}.
    \end{split}
\end{equation*}

Recall that by the relation of the distances in $M$, ${\rx\leq \contopy}$, so this last quotient in absolute value is less than or equal to $1$, so the image of $Tm_{y,0}$ lies inside the unit ball.

Thus, such operator $T$ satisfies that $\nu(T,m_{y,0})= \rx$.

Now, consider the unique operator $T$ defined by the previous images of $m_{x,0}$ and $m_{y,0}$. Then, the image of $m_{x,y}$ is completely determined by linearity, as
\begin{equation*}
    Tm_{x,y}= \frac{\ppx}{\dxy} Tm_{x,0} -  \frac{\ppy}{\dxy} Tm_{y,0}.
\end{equation*}

After fully developing the evaluation against $\pp_x$, we obtain
\begin{align*}
        \la \pp_x ,  Tm_{x,y} \ra &= \frac{\ppx}{\dxy} \la \pp_x, Tm_{x,0}\ra  -  \frac{\ppy}{\dxy} \la \pp_x, Tm_{y,0}\ra \\
        &= \rx \Big( 1- \frac{2\ppx}{\dxy} \frac{\grox}{\groy}\Big),
\end{align*}

which, as $0\leq \frac{2\ppx}{\dxy} \frac{\grox}{\groy}\leq 1$, implies that $| \la \pp_x ,  Tm_{x,y} \ra |\leq \rx$.

Analogously, evaluating against $\pp_y$,  we deduce that 
\begin{equation*}
    \begin{split}
        \la \pp_y ,  Tm_{x,y} \ra &= \frac{\ppx}{\dxy} \la \pp_y, Tm_{x,0}\ra  -  \frac{\ppy}{\dxy} \la \pp_y, Tm_{y,0}\ra \\
        &=  - \frac{\ppx}{\dxy} +  \frac{\ppy}{\dxy}\rx = -\rx.
    \end{split}
\end{equation*}

Therefore, $\nu(T)=\nu(T,m_{x,y})= \rx$. The last thing to do is check that $T$ is a norm-one operator by showing that $T m_{x,y}$ lies inside the ball. Since the computations for the other functionals are already done, we only need to study its evaluation against $\pp_0$. Thus, 
\begin{equation*}
    \begin{split}
        |\la \pp_0 ,  Tm_{x,y} \ra| &= \Big|\frac{\ppx}{\dxy} \la \pp_0, Tm_{x,0}\ra  -  \frac{\ppy}{\dxy} \la \pp_0, Tm_{y,0}\ra \Big| \\
        &=  \Big| \frac{\ppx}{\dxy} \Big( 1 - \aa_x \dfrac{\groy}{\dxy}\Big) +    \frac{\ppy}{\dxy} \rx \Big|\\
        &\leq \rx \frac{\ppx +\ppy}{\dxy} = \frac{\ppx}{\dxy} \frac{\ppx+\ppy}{\dxy + \ppy}\leq 1,
    \end{split}
\end{equation*}
which finishes the proof.
\end{proof}

\subsection{Operator for the optimal contribution}\label{subsec:op-contop}

The following proposition contains the explicit construction of the operator $T$ satisfying that  $\nu(T)= \contopx$ when $\contopx \geq \rx $.

\begin{proposition}\label{prop:constr-op-contop}
Let $M=(\{x,y,0\},d)$ be a canonical triangle 
and \[\contopx= \max\{\contopx, \rx \}.\] Then, there exist $T \in S_{\mathcal{L}(\free{M})}$ such that $\nu(T)= \contopx$.
\end{proposition}

\begin{proof}
 Once again, we will construct the operator by giving the image of $m_{x,0}$ and  $m_{y,0}$.
 
 First, let 
 \begin{equation*}
  Tm_{x,0}= \lambda_x m_{y,x}+ (1-\lambda_x) m_{y,0}
 \end{equation*}
where $\lambda_x$ is the optimal $\lambda$ for $m_{x,0}$ (see Lemma \ref{lemma:contop-bound} or \ref{lemma:opt-y}). The analogous version of Lemma \ref{lemma:opt-y} for the molecule $m_{x,0}$ would provide that
\begin{align*}
\begin{split}
    \la \pp_0,Tm_{x,0} \ra&= \contopx;\\
      \la \pp_x,Tm_{x,0} \ra&= \contopx;\\
         \la \pp_y,Tm_{x,0} \ra&=-1.
    \end{split}
\end{align*}

So, $\nu(T,m_{x,0})=\contopx$, and $\|T m_{x,0}\|=1$.

 Now, take
\begin{equation*}
    Tm_{y,0}:= \dfrac{\contopx}{\contopy} (\lambda_y m_{y,x}- (1-\lambda_y)m_{x,0}).
\end{equation*}

This image is a scaling of the vector given in Lemma \ref{lemma:opt-y}, so we deduce that
\begin{equation*}
\begin{split}
    \la \pp_0,T m_{y,0} \ra&= -\contopx;\\
      \la \pp_y,T m_{y,0} \ra&= -\contopx;\\
         \la \pp_x,T m_{y,0} \ra&=\dfrac{\contopx}{\contopy}.
    \end{split}
\end{equation*}

Recall that the order relation for the optimal contributions is the reverse of the order relation for the distances---see Lemma \ref{lemma:contop-order}---so, $\contopx \leq \contopy$. Thus, ${\nu(T,m_{y,0})=\contopx}$, and $\|Tm_{y,0}\|\leq 1$.

Finally, the image of $m_{x,y}$ is completely determined by the linearity of $T$ as 
\begin{equation*}
    Tm_{x,y}= \frac{\ppx}{\dxy} Tm_{x,0} -  \frac{\ppy}{\dxy} Tm_{y,0}.
\end{equation*}

Then, we compute all the evaluations. First,
\begin{equation*}
\begin{split}
  - \la \pp_x, Tm_{x,y}\ra &= -\Big(\frac{\ppx}{\dxy} \la \pp_x, Tm_{x,0} \ra-  \frac{\ppy}{\dxy}\la \pp_x, Tm_{y,0}\ra\Big) \\
   &= -\frac{\ppx}{\dxy} \contopx +  \frac{\ppy}{\dxy}\frac{\contopx}{\contopy} \\
   &= \frac{\contopx}{\dxy} \Big( \frac{\ppy}{\contopy} -\ppx\Big) \leq \contopx
   \end{split}
\end{equation*}

Last inequality is deduced from the fact that $\contopy \geq \ry$---see Lemma \ref{lemma:order-on-extremes}---and this expression is equivalent to 
\begin{align*}
  \dfrac{\ppy}{\contopy}-\ppx \leq\dxy .
\end{align*}
 Analogously, 
\begin{equation*}
   \la \pp_x, Tm_{x,y}\ra \leq \frac{\ppx}{\dxy}\contopx \leq \contopx.
\end{equation*}
So, $| \la g_x, Tm_{x,y}\ra|\leq \contopx$.
 
 We proceed in a similar way with the functional $\pp_y$:
 \begin{equation*}
 \begin{split}
   -\la \pp_y, Tm_{x,y} \ra & =  -\Big(\frac{\ppx}{\dxy} \la \pp_y, Tm_{x,0} \ra-  \frac{\ppy}{\dxy}\la \pp_y, Tm_{y,0}\ra\Big)\\
   &= \frac{\ppx}{\dxy} - \frac{\ppy}{\dxy} \contopx,
    \end{split}
 \end{equation*}
 which is less than or equal to $\contopx$ if and only if 
 \begin{equation*}
     \frac{\ppx}{\dxy} \leq (1+ \frac{\ppy}{\dxy})\contopx,
 \end{equation*}
which is, once again, equivalent to the hypothesis $\rx \leq \contopx$.
  
  For the other inequality, we just check that
\begin{equation*}
   - \la \pp_y, Tm_{x,y} \ra = \frac{\ppx}{\dxy} - \frac{\ppy}{\dxy} \contopx
 \end{equation*}
 
is greater than or equal to $-\contopx$ if and only if 
 \begin{equation*}
 \begin{split}
     \frac{\ppx}{\dxy}&\geq \Big( \frac{\ppy}{\dxy} -1\Big) \contopx,
 \end{split}
 \end{equation*}
which is trivially satisfied since the left-hand side is positive while the right-hand side is non-positive. In conclusion, $|\la \pp_y, Tm_{x,y} \ra|\leq \contopx$, which implies that ${\nu(T,m_{x,y})\leq \contopx}$.

Finally, we compute the evaluation against the last functional $\pp_0$ in order to check that $Tm_{x,y}$ is inside the unit ball. 
\begin{equation*}
\begin{split}
   \la \pp_0, Tm_{x,y}\ra &= \frac{\ppx}{\dxy} \la \pp_0, Tm_{x,0} \ra-  \frac{\ppy}{\dxy}\la \pp_0, Tm_{y,0}\ra \\
   &= \frac{\ppx}{\dxy} \contopx+  \frac{\ppy}{\dxy} \contopx = \frac{\contopx}{\rxy}.
   \end{split}
\end{equation*}

Recall that under our assumptions $\rxy \geq \contopx$---see Lemma \ref{lemma:order-on-extremes}, this last quotient is less than or equal to $1$.  

This completes the construction of the operator $T$ with $\|T\|=1$ and ${\nu(T)=\contopx}$.
\end{proof}

\section{Some applications}\label{sec:app}

Now we are ready to state Theorem \ref{thm:num-ind-2D} that can be rewritten using the notation developed in previous sections.  

\begin{theorem}\label{thm:num-ind-2D-restated}
Let $M=(\{x,y,0\},d)$ be metric space of three points with $\dxy \geq \ppx \geq \ppy$.
Then:
\begin{itemize}
\item if $M$ is aligned, then ${n(\free{M})=1}$;
\item  otherwise, if $M$ is a triangle, then 
    \begin{equation*}
        n(\free{M})= \max \left\{  \contopx, \rx \right\}.
    \end{equation*}
\end{itemize}
\end{theorem}

\begin{corollary}\label{cor:aligned}
 Let $M$ a metric space of three points. Then, $n(\free{M})=1$ if and only if the three points of $M$ are aligned. 
\end{corollary}

Using this theorem it is easy to get some insights on the relation between the different types of triangles and the numerical index. The first thing to notice is that if $M$ is equilateral, then both values in the maximum are equal to $\frac{1}{2}$, and combining this with Proposition \ref{prop:2D-1/2}, we reach that the equilateral case is, between all the triangles, the only one providing the minimal numerical index.

\begin{corollary}\label{cor:equil}
 Let $M$ a metric space of three points. Then, $n(\free{M})=\frac{1}{2}$ if and only if $M$ is an equilateral triangle.
\end{corollary}

Assuming that $M$ is an isosceles triangle---i.e., two sides are equal---the formula is even simpler. Indeed, if $M=(\{x,y,0\},d)$ with $\ppx=\ppy$, then ${\contopx=\frac{\ppx}{3\ppx-\dxy}}$, $\rx= \frac{\ppx}{\ppx+\dxy}$. Moreover, ${\contopx\leq \rx}$ if and only if ${\dxy\leq \ppx}$, which allows as to give the following corollaries. 

\begin{corollary}\label{cor:long-isos}
Let $M=(\{x,y,0\},d)$ be an isosceles triangle such that ${\ppx=\ppy\geq \dxy}$. Then then 
\[n(\free{M})= \frac{\ppx}{\ppx+\dxy}.\] 
In particular, $n(\free{M})\in [\frac{1}{2},1)$, and it can take any value of the interval.
\end{corollary}
\begin{corollary}\label{cor:fat-isos}
Let $M=(\{x,y,0\},d)$ be an isosceles triangle such that ${\ppx=\ppy\leq \dxy}$. Then \[n(\free{M})= \frac{\ppx}{3\ppx-\dxy}.\] In particular, $n(\free{M})\in [\frac{1}{2},1)$, and it can take any value of the interval.
\end{corollary}

M. Martin kindly pointed us to the relation between the numerical index of a $2$-dimensional Lipschitz-free space and \cite[Theorem 1]{MarMer07}. Using Corollaries \ref{cor:long-isos} and \ref{cor:fat-isos}, we can recover the numerical index of the family of hexagonal norms in \cite[Theorem 1]{MarMer07} and the other way around. It is easy to prove that a hexagonal norm in $\RR^2$ is covered by \cite[Theorem 1]{MarMer07} if and only if the space is isometrically isomorphic to $\free{M}$ for a certain isosceles triangle $M$.
\begin{remark}
\rm Notice that the formula obtained in Theorem \ref{thm:num-ind-2D-restated} allows to compute the numerical index for every $2$-dimensional hexagonal norm, as such a Banach space is isometrically isomorphic to $\free{M}$ for some triangle $M$. 
\eor 
\end{remark}


We can even deduce some consequences for infinite-dimensional Lipschitz-free spaces. Considering $A$ a metric space and a subset $M\subset A$, then $\free{M}$ is canonically embedded as a linear subspace of $\free{A}$. A straightforward consequence of this fact is the following result.

\begin{theorem}\label{thm:R2-contains-all}
    Let $A\subset (\RR^2,\|\cdot\|_2)$ be any subset with non-empty interior. Then, $\free{A}$ is a separable infinite-dimensional Lipschitz-free space such that, for every $\aa \in [\frac{1}{2},1]$, it contains a $2$-dimensional subspace $Y_\aa$ with $n(Y_\aa)=\aa$.
\end{theorem}

This may even allow some other application. For instance, as a consequence of the Theorem above and Corollary \ref{cor:fat-isos}, we are able to provide the---to our understanding, curious---result below. Its proof will appear elsewhere.

\begin{theorem}\label{thm-num-ind-dynamics}
Let $A\subset (\RR^2,\|\cdot\|_2)$ be any subset with non-empty interior. Then, there exist an operator $T\in \mathcal{L}(\free{A})$ and an uncountable infinite family of $2$-dimensional subspaces $\{Y_i:i\in I\}$ of $\free{A}$ such that, for every $i\in I$, the set $\{n(T^n(Y_i)):n \in \NN \}$ is dense in $[\frac{1}{2},1]$.
\end{theorem}


 \section*{Acknowledgements}
The authors are grateful to M. Martin for his kind suggestions and commentaries, which improved the final version of the present document.
The authors were partially supported by the Universitat Polit\`ecnica de Val\`encia (Spain) and by grant PID2021-122126NB-C33 funded by MCIN/AEI/10.13039/ 501100011033 and by “ERDF A way of making Europe”. The first author was also supported by Generalitat Valenciana (through Project PROMETEU/2021/070 and the predoctoral fellowship CIACIF/2021/378) and by MCIN/AEI/10.13039/501100011033 (Project PID2019-105011GB). The second author was also supported by Fundación Séneca, Región de Murcia (Grant 19368/PI/14). The third author was also supported by AEI/FEDER (Project MTM2017-83262-C2-1-P of Ministerio de Economía y Competitividad).

 \printbibliography

  \end{document}